\documentclass[11pt]{amsart}

\usepackage{amsfonts}
\usepackage{amsmath}
\usepackage{amssymb, color}
\usepackage{amscd}
\usepackage[mathscr]{eucal}
\usepackage{url}
\usepackage{enumerate}
\usepackage{comment}

\calclayout

\allowdisplaybreaks[1]

\theoremstyle{plain}

\newtheorem{theorem}{Theorem}[section]

\newtheorem{prop}[theorem]{Proposition}
\newtheorem{cor}[theorem]{Corollary}
\theoremstyle{definition}

\newtheorem{remark}[theorem]{Remark}
\newtheorem{ex}[theorem]{Example}

\numberwithin{equation}{section}

\makeatletter
\@namedef{subjclassname@2020}{%
  \textup{2020} Mathematics Subject Classification}
\makeatother

\begin{document}

\title[Self-conjugate $t$-core partitions and applications]{Self-conjugate $t$-core partitions and applications}

\author{Madeline Locus Dawsey and Benjamin Sharp}

\address{Madeline Locus Dawsey, Department of Mathematics, The University of Texas at Tyler, Tyler, TX 75799, USA}

\address{Benjamin Sharp, Department of Mathematics, The University of Texas at Tyler, Tyler, TX 75799, USA}

\email{mdawsey@uttyler.edu}

\email{bsharp4@patriots.uttyler.edu}

\thanks{The first author is grateful for the support of an AMS-Simons Travel Grant and an internal grant from the University of Texas at Tyler.}



\begin{abstract}
Partition theory abounds with bijections between different types of partitions.  One of the most famous partition bijections maps each self-conjugate partition of a positive integer $n$ to a partition of $n$ into distinct odd parts, and vice versa.  Here we prove new necessary and sufficient conditions for a self-conjugate partition to be $t$-core, in terms of only the parts of the corresponding partition into distinct odd parts, by proving a new hook length formula.  Corollaries of these results include new applications of $t$-core self-conjugate partitions to subsets of the natural numbers, due to the recent investigation of a new partition statistic called the supernorm by the first author, Just, and Schneider, as well as many results on $t$-cores by Bringmann, Kane, Males, Ono, Raji, and others.  We provide several examples of these applications, one of which gives a new formula for certain families of Hurwitz class numbers.
\end{abstract}

\maketitle

%%%%%%%%%%%%%%%%%%%%%%%%%%%%%

\section{Introduction and statement of results}

A \emph{partition} $\lambda$ of a positive integer $n$ is a non-increasing sequence $\lambda=\left(\lambda_1,\dots,\lambda_k\right)$ of positive integers which sum to $n$.  The integers $\lambda_1,\dots,\lambda_k$ are called the \emph{parts} of the partition.  The sum $n$ is called the \emph{size} of the partition and denoted $|\lambda|$, and the number of parts is called the \emph{length} of the partition and denoted $\ell(\lambda)$.  The notation $\lambda\vdash n$ is also common notation for a partition $\lambda$ of size $n$.  Despite their simplicity, integer partitions have many, many deep applications in varied fields, including algebraic and analytic number theory as well as combinatorics.  For example, see \cite{A,AE,HW} for background information on partitions and \cite{HR,Rademacher,R1,R2,R3} for several famous results on the number of partitions of any fixed size $n$.  Some applications arise from certain sets of restricted partitions and bijections between those sets.  The most classical example of a partition bijection is due to Euler:
\begin{align*}
&\textit{The number of partitions of a positive integer }n\textit{ into odd parts}\\
&\textit{is equal to the number of partitions of }n\textit{ into distinct parts}.
\end{align*}
This particular partition bijection can be proved by cleverly constructing an explicit bijection between the set of partitions into odd parts and the set of partitions into distinct parts which uses uniqueness of the binary representations of the parts.  Constructive methods of proof are common for partition bijections in general.

Each partition $\lambda=\left(\lambda_1,\lambda_2,\dots,\lambda_k\right)\vdash n$ has an associated \emph{Young diagram}, which is a diagram of $n$ total boxes arranged in $\ell(\lambda)=k$ rows such that row $i$ contains $\lambda_i$ boxes, for all $1\leq i\leq k$.  A \emph{self-conjugate partition} of a positive integer $n$ is a partition whose corresponding Young diagram is symmetric by reflection across the main diagonal.  Throughout, $\gamma$ will be used to denote a self-conjugate partition.  Another famous example of a partition bijection is:
\begin{align}\label{bijection}
&\textit{The number of self-conjugate partitions of a positive integer }n\\
&\textit{is equal to the number of partitions of }n\textit{ into distinct odd parts}.\nonumber
\end{align}
The proof of this bijection is simple using only the Young diagram of the self-conjugate partition.  In particular, the main ingredient involves what are called the \emph{hook lengths} (alternatively, \emph{hook numbers}) of the boxes along the main diagonal.  Let $(i,j)$ denote the box in row $i$ and column $j$ of  the Young diagram of a partition $\lambda$.  The hook length of box $(i,j)$ is defined to be the number of boxes below $(i,j)$, to the right of $(i,j)$, and at the box $(i,j)$ itself, and it is denoted $h_\lambda(i,j)$.  See Figure 1 below for an example, and notice that this is also a self-conjugate partition.

\begin{center}
\begin{itemize}\itemsep-1.5pt
\item[]\hspace{5cm}\fbox{\parbox{.3cm}{\hspace{-2.42pt}13}}\hspace{-.5pt}\fbox{\parbox{.3cm}{\hspace{-2.42pt}10}}\hspace{-.5pt}\fbox{\parbox{.3cm}{8}}\hspace{-.5pt}\fbox{\parbox{.3cm}{7}}\hspace{-.5pt}\fbox{\parbox{.3cm}{4}}\hspace{-.5pt}\fbox{\parbox{.3cm}{2}}\hspace{-.5pt}\fbox{\parbox{.3cm}{1}}\\
\item[]\hspace{5cm}\fbox{\parbox{.3cm}{\hspace{-2.42pt}10}}\hspace{-.5pt}\fbox{\parbox{.3cm}{7}}\hspace{-.5pt}\fbox{\parbox{.3cm}{5}}\hspace{-.5pt}\fbox{\parbox{.3cm}{4}}\hspace{-.5pt}\fbox{\parbox{.3cm}{1}}\\
\item[]\hspace{5cm}\fbox{\parbox{.3cm}{8}}\hspace{-.5pt}\fbox{\parbox{.3cm}{5}}\hspace{-.5pt}\fbox{\parbox{.3cm}{3}}\hspace{-.5pt}\fbox{\parbox{.3cm}{2}}\\
\item[]\hspace{5cm}\fbox{\parbox{.3cm}{7}}\hspace{-.5pt}\fbox{\parbox{.3cm}{4}}\hspace{-.5pt}\fbox{\parbox{.3cm}{2}}\hspace{-.5pt}\fbox{\parbox{.3cm}{1}}\\
\item[]\hspace{5cm}\fbox{\parbox{.3cm}{4}}\hspace{-.5pt}\fbox{\parbox{.3cm}{1}}\\
\item[]\hspace{5cm}\fbox{\parbox{.3cm}{2}}\\
\item[]\hspace{5cm}\fbox{\parbox{.3cm}{1}}
\end{itemize}
\vspace{.2cm}
Figure 1: Young diagram of $\gamma=(7,5,4,4,2,1,1)$ with hook lengths
\end{center}

For example, $h_\gamma(2,3)=5$ and $h_\gamma(3,1)=8$.

A self-conjugate partition $\gamma\vdash n$ corresponds to a partition $\lambda=\left(\lambda_1,\lambda_2,\dots,\lambda_k\right)\vdash n$ into distinct odd parts via the partition bijection \eqref{bijection} mentioned above if the hook lengths of $\gamma$ along the main diagonal yield the parts of $\lambda$: $h_\gamma(i,i)=\lambda_i$ for each $1\leq i\leq k$.  In Figure 1, the partition $\lambda$ into distinct odd parts that corresponds in such a way to the given self-conjugate partition $\gamma$ is the partition $\lambda=(13,7,3,1)$.  Throughout, when we say that each self-conjugate partition $\gamma$ corresponds to a partition $\lambda$ into distinct odd parts, we mean that $\gamma$ and $\lambda$ map to each other via this bijection.

Divisibility properties of hook lengths have been studied extensively, although they are difficult to prove in general.  A partition is called \emph{$t$-core} if none of its hook lengths are divisible by $t$.  Results have been proven previously to determine exactly when a partition is $t$-core, and these results typically depend on Frobenius symbols and beta-sets (for example, see Olsson's work \cite{O1,O2} and Vandehey's undergraduate thesis \cite{V}).  Applications of hook length divisibility results have been developed as well (see work of Garvan, Granville, Ono, and Sze \cite{Garvan,GO,Ono1,Ono2,OS} for background on the topic), some of which are specifically for self-conjugate partitions (for example, see work of Bringmann, Kane, Males, Ono, and Raji \cite{BKM,OR}).  These applications for self-conjugate partitions will be provided in Section \ref{section_hooks}.

As of yet, it seems that no results for $t$-core self-conjugate partitions depend solely on the corresponding partition into distinct odd parts, but such results could deepen these applications when paired with recent work of the first author, Just, and Schneider \cite{DJS} on a new partition statistic called the \emph{supernorm} of a partition (see Section \ref{sec_Hurwitz} for more details).  In particular, the supernorm maps each partition to the prime factorization of a positive integer which is uniquely determined by the parts of the partition.  For this reason, taking advantage of the corresponding partition into distinct odd parts provides enough information to achieve meaningful results about prime factorizations of integers via the supernorm map, whereas formulas using Frobenius symbols and beta-sets do not.

Here we exploit the partition bijection \eqref{bijection} to prove a new formula for hook lengths of self-conjugate partitions based only on the parts of the corresponding partitions into distinct odd parts.  This formula immediately gives necessary and sufficient conditions for when a self-conjugate partition is $t$-core.  First, we state a simple preliminary proposition regarding hook length divisibility related to gaps between parts of partitions to give an idea of the combinatorial flavor of the results and applications discussed in this paper.

\begin{prop}
Let $\gamma$ be a self-conjugate partition, and let $\lambda=\left(\lambda_1,\lambda_2,\dots,\lambda_k\right)$ be the corresponding partition into distinct odd parts.  If $\lambda_i-\lambda_{i+1}\geq2(t+1)$ for any $1\leq i<k$, then $\gamma$ is not t-core.
\end{prop}

\begin{proof}
Suppose $\lambda_i-\lambda_{i+1}\geq2(t+1)$ for some $1\leq i<k$.  Applying the bijection $\lambda\mapsto\gamma$, $\lambda_i$ and $\lambda_{i+1}$ become consecutive hooks of $\gamma$ with corners located along the main diagonal.  Once $\lambda_i$ (resp. $\lambda_{i+1}$) is folded in the middle by the bijection, each leg has $\frac{1}{2}\left(\lambda_i-1\right)$ boxes $\big($resp. $\frac{1}{2}\left(\lambda_{i+1}-1\right)$ boxes$\big)$, excluding the corner box.  By assumption, $\lambda_i\geq\lambda_{i+1}+2(t+1)$, and therefore each leg of $\lambda_i$ has at least $$\frac{1}{2}\left(\lambda_{i+1}+2(t+1)-1\right)=\frac{1}{2}\left(\lambda_{i+1}+1\right)+t$$ boxes, excluding the corner box.  This means that each leg of $\lambda_i$ has at least $$\left(\frac{1}{2}\left(\lambda_{i+1}+1\right)+t\right)-\left(\frac{1}{2}\left(\lambda_{i+1}-1\right)\right)=t+1$$ more boxes than each leg of $\lambda_{i+1}$.  Since $\lambda_i$ is centered at box $(i,i)$ and $\lambda_{i+1}$ is centered at box $(i+1,i+1)$, $\lambda_i$ appears in at least $(t+1)-1=t$ more rows of the Young diagram than $\lambda_{i+1}$.  Thus, the hook length of box $$\left(i+\frac{1}{2}\left(\lambda_i-1\right)-(t-1),i\right)$$ is exactly $t$, and so $\gamma$ is not $t$-core.
\end{proof}

More generally, one can calculate all of the hook lengths of a self-conjugate partition using \eqref{bijection}.  The following theorem provides a new way to calculate these hook lengths.

\begin{theorem}\label{hook_thm}
Let $\gamma$ be a self-conjugate partition, and let $\lambda=\left(\lambda_1,\lambda_2,\dots,\lambda_k\right)$ be the corresponding partition into distinct odd parts.  Let $i,j$ be integers such that $(i,j)$ is a box in the Young diagram of $\gamma$, and assume without loss of generality that $i\geq j$.  Then the hook lengths of $\gamma$ are given by $$h_\gamma(i,j)=h_\gamma(j,i)=\begin{cases}
\dfrac{1}{2}\left(\lambda_i+\lambda_j\right),&\mbox{if }i\leq k\\
\\[1pt]
\dfrac{1}{2}\left(\lambda_j+1\right)+j-i-1\\
\mbox{}\hspace{0.25cm}+\#\left\{j\leq m\leq k:\lambda_m\geq 2i-2m+1\right\},&\mbox{if }i>k.
\end{cases}$$
\end{theorem}

\begin{remark}
Theorem \ref{hook_thm} for $i\leq k$ was stated without proof in \cite[Lemma 2.1]{HN}.
\end{remark}

\begin{remark}\label{remark}
The above formula for $i\leq k$ gives the hook lengths of $\gamma$ within the \emph{Durfee square} (the largest square which can fit in the top left corner of the Young diagram), whereas the formula for $i>k$ gives the hook lengths of $\gamma$ outside of the Durfee square (see Section \ref{sec_proofs} for details).  The fact that $h_\gamma(i,j)=h_\gamma(j,i)$ is discussed in further detail in Section \ref{section_sc_partitions}.
\end{remark}

This hook length formula gives a process by which one can immediately identify the self-conjugate partitions which are $t$-core, although in practice this process is quite inefficient as it requires the calculation of at least half of the hook lengths in the Young diagram.

\begin{cor}\label{corollary}
Let $\gamma$ be a self-conjugate partition, and let $\lambda=\left(\lambda_1,\lambda_2,\dots,\lambda_k\right)$ be the corresponding partition into distinct odd parts.  Then $\gamma$ is $t$-core if and only if $$\frac{1}{2}\left(\lambda_i+\lambda_j\right)\not\equiv0\pmod{t}$$ for all pairs of integers $(i,j)$ with $1\leq j\leq i\leq k$, and $$\frac{1}{2}\left(\lambda_j+1\right)+j-i-1+\#\left\{j\leq m\leq k:\lambda_m\geq 2i-2m+1\right\}\not\equiv0\pmod{t}$$ for all pairs of integers $(i,j)$ with $1\leq j\leq k<i\leq\frac{1}{2}\left(\lambda_1+1\right)$ and $\lambda_j\geq 2i-2j+1$.
\end{cor}

The proof of Corollary \ref{corollary} is immediate from Theorem \ref{hook_thm} and the definition of $t$-core partitions.  The requirement that $\lambda_j\geq 2i-2j+1$ guarantees that the box $(i,j)$ is contained in the Young diagram of $\gamma$; this is clear from equation \eqref{iboxes} and the preceding text, which identify the columns of the Young diagram that contain at least $i$ boxes.  Also note that $\frac{1}{2}\left(\lambda_1+1\right)=\ell(\gamma)$.

The following is an example of a self-conjugate partition which we determine to be $6$-core but not $7$-core, based on Corollary \ref{corollary}.

\begin{ex}
Define the self-conjugate partition $\gamma$ and the corresponding partition $\lambda$ into distinct odd parts by $$\gamma=(7,5,4,4,2,1,1)\hspace{.5cm}\text{and}\hspace{.5cm}\lambda=(13,7,3,1).$$  Notice that these are the same partitions $\gamma$ and $\lambda$ as in Figure 1.  The pairs $(i,j)$, $1\leq j\leq i\leq k$, and the respective values of $\frac{1}{2}\left(\lambda_i+\lambda_j\right)$ are
\begin{align*}
(1,1)&:\hspace{.5cm}\tfrac{1}{2}(13+13)=13,\\
(2,1)&:\hspace{.5cm}\tfrac{1}{2}(7+13)=10,\\
(3,1)&:\hspace{.5cm}\tfrac{1}{2}(3+13)=8,\\
(4,1)&:\hspace{.5cm}\tfrac{1}{2}(1+13)=7,\\
(2,2)&:\hspace{.5cm}\tfrac{1}{2}(7+7)=7,\\
(3,2)&:\hspace{.5cm}\tfrac{1}{2}(3+7)=5,\\
(4,2)&:\hspace{.5cm}\tfrac{1}{2}(1+7)=4,\\
(3,3)&:\hspace{.5cm}\tfrac{1}{2}(3+3)=3,\\
(4,3)&:\hspace{.5cm}\tfrac{1}{2}(1+3)=2,\\
(4,4)&:\hspace{.5cm}\tfrac{1}{2}(1+1)=1.
\end{align*}
We see that none of the values above are divisible by 6.  The pairs $(i,j)$ with $1\leq j\leq k<i\leq\frac{1}{2}\left(\lambda_1+1\right)$ and $\lambda_j\geq 2i-2j+1$ and the respective values of $$\frac{1}{2}\left(\lambda_j+1\right)+j-i-1+\#\left\{j\leq m\leq k:\lambda_m\geq 2i-2m+1\right\}$$ are
\begin{align*}
(5,1)&:\left(7+1-5-1\right)+\#\{1\leq m\leq 4:\lambda_m\geq11-2m\}=2+2=4,\\
(6,1)&:\left(7+1-6-1\right)+\#\{1\leq m\leq 4:\lambda_m\geq13-2m\}=1+1=2,\\
(7,1)&:\left(7+1-7-1\right)+\#\{1\leq m\leq 4:\lambda_m\geq15-2m\}=0+1=1,\\
(5,2)&:\left(4+2-5-1\right)+\#\{2\leq m\leq 4:\lambda_m\geq11-2m\}=0+1=1.
\end{align*}
We also see that none of these values are divisible by 6.  Therefore, $\gamma$ is $6$-core.  To show that $\gamma$ is not $7$-core, it suffices to identify one pair $(i,j)$ in either of the cases above such that the corresponding calculation is divisible by 7.  For example, we see above that the hook length of box $(i,j)=(4,1)$ is $\frac{1}{2}(1+13)=7$.  Therefore, $\gamma$ is not $7$-core.  See Figure 1 above with hook lengths included for an illustration.  Notice also that the hook lengths below and including the main diagonal (equivalently, above and including the main diagonal, by symmetry) are equal to the values calculated from Corollary \ref{corollary}, as mentioned in Remark \ref{remark}.
\end{ex}

\begin{comment}
The following theorem also gives a necessary and sufficient condition for a self-conjugate partition to be $t$-core.

\begin{theorem}\label{check_thm}
Let $\gamma$ be a self-conjugate partition, and let $\lambda=\left(\lambda_1,\lambda_2,\dots, \lambda_k\right)$ be the corresponding partition into distinct odd parts.  Then $\gamma$ is $t$-core if and only if $$\frac{1}{2}\left(\lambda_i+\lambda_j\right)\not\equiv0\pmod{t}$$ for all pairs $(i,j)$ such that $1\leq j \leq i \leq k$, and one of the following holds for each pair $(i,j)$ such that $1\leq j\leq k<i\leq\frac{1}{2}\left(\lambda_1+1\right)$ with $\lambda_j\geq 2i-2j+1$:
\begin{enumerate}
\item $tk-m+j-1>k$,
\item $\lambda_{tk-m+j-1}<\lambda_j-2tk+2j-2$,
\item $\lambda_{tk-m+j}\geq\lambda_j-2tk+2j-4$.
\end{enumerate}
\end{theorem}
%The issues with this theorem are: the variable $m$ is never defined, and the proof does not completely justify the theorem.  I can't verify that it's true with the proof provided, so I can't include it in the paper.
\end{comment}

Although Corollary \ref{corollary} requires calculations for at least half of the boxes in the Young diagram -- that is, at least $|\gamma|/2$ different hook length calculations -- it is a new formula which does not rely on Frobenius symbols and only requires parts of partitions.

In Section \ref{section_partitions}, we give more details about self-conjugate partitions and hook lengths, as well as some known applications of $t$-core partitions.  In Section \ref{sec_proofs}, we prove Theorem \ref{hook_thm}.  Finally, in Section \ref{sec_Hurwitz}, we provide some background on the recently defined supernorm map and a few explicit examples of resulting applications.

%%%%%%%%%%%%%%%%%%%%%%%%%%%%%

\section{Partitions and $t$-cores}\label{section_partitions}

Here we recall background information and known results on self-conjugate partitions and $t$-core partitions.  In Section \ref{section_sc_partitions}, we give more details about self-conjugate partitions and their hook lengths.  In Section \ref{section_hooks}, we recall some known results on self-conjugate $t$-core partitions.

%%%%%%%%%%%%%%%%%%%%%%%%%%%%%

\subsection{Self-conjugate partitions}\label{section_sc_partitions}

As mentioned in the previous section, a self-conjugate partition is a partition whose Young diagram is symmetric by reflection across the main diagonal (from the top left corner).  In other words, a partition $\gamma$ is self-conjugate if the \emph{conjugate} of $\gamma$, which is the new partition obtained by swapping the rows and columns of its Young diagram, is the same partition as $\gamma$ itself.

The hook lengths of a self-conjugate partition $\gamma$, which we recall count the number of boxes extending downward and to the right in a hook within the Young diagram of $\gamma$, are symmetric in self-conjugate partitions.  Namely, we have that if $\gamma$ is self-conjugate, then $h_\gamma(i,j)=h_\gamma(j,i)$ for all boxes $(i,j)$ in the Young diagram of $\gamma$.  It is straightforward to justify this fact using the symmetry of $\gamma$ itself.  Since $\gamma$ is symmetric by reflection across the main diagonal, the hook length of each box must also be the same for $\gamma$ as for its conjugate (which is also equal to $\gamma$).

Hook length symmetry does simplify our calculations somewhat, as we must only calculate the hook lengths of a little over half of the boxes in the Young diagram in order to use Corollary \ref{corollary} to determine whether a self-conjugate partition is $t$-core.  Specifically, the number of hook lengths one must calculate using Theorem \ref{hook_thm} is the number of boxes below the Durfee square (or to the right of the Durfee square) plus the number of boxes whose intersection with the half of the area in the Durfee square below the main diagonal (resp. above the main diagonal) is nonempty.  Recall the partition bijection \eqref{bijection} between self-conjugate partitions and partitions into distinct odd parts.  For a self-conjugate partition $\gamma\vdash n$ corresponding via this bijection to a partition $\lambda$ into distinct odd parts, the discussion above leads to $$\frac{n-\ell(\lambda)^2}{2}+\left(\frac{\ell(\lambda)^2}{2}+\frac{\ell(\lambda)}{2}\right)=\frac{n+\ell(\lambda)}{2}$$ hook length calculations, since $\ell(\lambda)^2/2$ only counts half of each box along the main diagonal of $\gamma$ within the Durfee square.

%%%%%%%%%%%%%%%%%%%%%%%%%%%%%

\subsection{Self-conjugate $t$-core partitions}\label{section_hooks}

In this section, we recall known results on the number of self-conjugate $2$-core, $3$-core, and $7$-core partitions of any integer $n$.  These are the examples for which we will give applications of the supernorm map in Section \ref{sec_Hurwitz}.

We say a partition $\gamma$ of size $n$ is \emph{perfectly triangular} if $\gamma$ is of the form $$\gamma=(k,k-1,k-2,\dots,3,2,1).$$  Notice that for a partition $\gamma\vdash n$ of this form, $n$ must be a triangular number, i.e. $$n=k(k+1)/2=\sum_{i=1}^k i,$$ and $\gamma$ is self-conjugate.  This type of partition is referred to as a ``fancy triangle" in \cite{Rebecca}.  The following well-known theorem about $2$-core partitions can be found in \cite{Robbins}, for example, but we provide a proof below that uses Theorem \ref{hook_thm}.

\begin{theorem}\label{2core_thm}
Let $\gamma$ be a partition.  Then $\gamma$ is $2$-core if and only if $\gamma$ is perfectly triangular.
\end{theorem}

\begin{proof}
Suppose $\gamma=(k,k-1,k-2,\dots,3,2,1)$ is a perfectly triangular partition.  Then $\gamma$ is self-conjugate by \cite[Lemma 2.4]{Rebecca}.  Therefore, $\gamma$ corresponds to a partition $\lambda$ into distinct odd parts.  Since $h_\gamma(i,i)=\lambda_i$ for all $1\leq i\leq\ell(\lambda)$, the largest part of $\lambda$ is $\lambda_1=2k-1$.  The second largest part of $\lambda$ must be $\lambda_2=\lambda_1-4=2k-5$, since the hook with corner $(2,2)$ must fit inside the hook with corner $(1,1)$ in order for the parts of $\gamma$ to be non-increasing (meaning $h_\gamma(2,2)\leq h_\gamma(1,1)-2$) and must be exactly one additional box shorter in each leg in order for $\gamma$ to be perfectly triangular (meaning $h_\gamma(2,2)=h_\gamma(1,1)-4$).  By the same reasoning, each part of $\lambda$ must be four less than the previous part, so $\lambda$ is given by $$\lambda=\big(2k-1,2(k-2)-1,2(k-4)-1,\dots\big).$$  Note that if $k=2m-1$ is odd, then $2k-1=4m-3\equiv1\pmod{4}$, and so $$\ell(\lambda)=\left\lfloor\frac{2k-1}{4}\right\rfloor=\left\lfloor m-\frac{3}{4}\right\rfloor=m-1=\frac{k-1}{2}$$ and the smallest part of $\lambda$ is 1.  On the other hand, if $k=2m$ is even, then $2k-1=4m-1\equiv3\pmod{4}$, and so $$\ell(\lambda)=\left\lfloor\frac{2k-1}{4}\right\rfloor=\left\lfloor m-\frac{1}{4}\right\rfloor =m-1=\frac{k}{2}-1$$ and the smallest part of $\lambda$ is 3.  In either case, we use Theorem \ref{hook_thm} to verify that the hook lengths of $\gamma$ are odd.  The values of $\frac{1}{2}\left(\lambda_i+\lambda_j\right)$ for all pairs $(i,j)$, $1\leq j\leq i\leq\ell(\lambda)$, are of the form $$\frac{1}{2}\bigg(\big(2(k-2\ell)-1\big)+\big(2(k-2s)-1\big)\bigg)=\frac{1}{2}(4k-4\ell-4s-2)=2(k-\ell-s)-1$$ for integers $\ell,s\geq0$.  Therefore, the hook lengths within the Durfee square are all odd.  Now, consider the boxes in the Young diagram below the Durfee square.  These are the parts of $\gamma$ whose indices are at least $\ell(\lambda)=\frac{k-1}{2}$ if $k$ is odd (resp. $\ell(\lambda)=\frac{k}{2}-1$ if $k$ is even) and at most $\ell(\gamma)=k$.  These parts form a sub-partition (see \cite{Rebecca,S} for more details on sub-partitions) $\widehat{\gamma}$ of $\gamma$ of the form $$\widehat{\gamma}=\left(\left\lfloor\frac{k}{2}\right\rfloor,\left\lfloor\frac{k}{2}\right\rfloor-1,\dots,3,2,1\right),$$ which is itself a perfectly triangular partition.  We continue this process inductively, eliminating Durfee squares of successive sub-partitions until all that remains of $\gamma$ is the sub-partition $(1)$, whose one box has hook length 1.  Thus, all of the hook lengths of $\gamma$ are odd.

Conversely, suppose $\gamma$ is $2$-core.  Then each part of $\gamma$ must be at most one more than the next part, or else the second to last box in the row corresponding to the larger part will have a hook length of 2.  Each part of $\gamma$ must also be at least one more than the next part, because if at least two rows have the same number of boxes, then the last box in the second to last such row will have a hook length of 2.  Therefore, each part of $\gamma$ must be exactly one more than the next part, so $\gamma$ must be perfectly triangular.
\end{proof}

\begin{remark}
The above proof of Theorem \ref{2core_thm} uses techniques resulting from the connection between self-conjugate $t$-cores and partitions into distinct odd parts to emphasize their usefulness.  This particular proof can be made much simpler without using these techniques.
\end{remark}

In particular, there is a $2$-core partition of size $n$ if and only if $n$ is a triangular number, and in this case the $2$-core partition is unique: it is the perfectly triangular (self-conjugate) partition of size $n$.  If $\mathrm{sc}_t(n)$ denotes the number of self-conjugate $t$-core partitions of $n$, then we have that
\begin{equation}\label{2core_count}
\mathrm{sc}_2(n)=\begin{cases}
1,&\mbox{if }n=k(k+1)/2\text{ for some }k\geq1,\\
0,&\mbox{otherwise}.
\end{cases}
\end{equation}

Next, we turn to self-conjugate $3$-core partitions, which have been studied by Robbins \cite{Robbins}.

\begin{theorem}[Robbins, Theorem 7 of \cite{Robbins}]\label{3core_thm1}
There is a self-conjugate $3$-core partition of size $n$ if and only if there exists $r\geq1$ such that $n=r(3r\pm2)$.  If such a self-conjugate $3$-core partition exists, then it is unique.
\end{theorem}

The proof of Theorem 7 in \cite{Robbins} is constructive and provides the following nice, closed formula for the parts of the unique self-conjugate $3$-core partitions of size $n=r(3r\pm2)$, $r\geq1$.  For notational convenience, we use the frequency notation $\gamma=\left\langle1^{m_1},2^{m_2},3^{m_3},\dots\right\rangle$ for the partition $\gamma$, where $m_i$ is the multiplicity of the part $i$ in $\gamma$.

\begin{theorem}[Robbins, \cite{Robbins}]\label{3core_thm2}
The unique self-conjugate $3$-core partition of $n=r(3r-2)$, $r\geq1$, is $$\gamma_-=\left\langle1^2,2^2,\dots,(r-2)^2,(r-1)^2,r,r+2,\dots,3r-4,3r-2\right\rangle.$$  The unique self-conjugate $3$-core partition of $n=r(3r+2)$, $r\geq1$, is $$\gamma_+=\left\langle1^2,2^2,\dots,(r-2)^2,(r-1)^2,r^2,r+2,\dots,3r-2,3r\right\rangle.$$
\end{theorem}

Theorem \ref{3core_thm1} counts the number of self-conjugate $3$-core partitions of $n$:
\begin{equation}\label{3core_count}
\mathrm{sc}_3(n)=\begin{cases}
1,&\mbox{if }n=r(3r\pm2)\text{ for some }r\geq1,\\
0,&\mbox{otherwise}.
\end{cases}
\end{equation}

Counting self-conjugate $t$-core partitions of an integer $n$ becomes more difficult as $t$ gets larger, so there are very few explicit expressions for $\mathrm{sc}_t(n)$ for larger $t$.  The number of self-conjugate $7$-core partitions has also been studied by Ono and Raji, and by Bringmann, Kane, and Males, because of its interesting relationship to Hurwitz class numbers.  For a discriminant $-D<0$, the \emph{Hurwitz class number} $H(-D)$ is equal to the number of equivalence classes in $\mathrm{SL}_2(\mathbb{Z})$ of integral, binary quadratic forms of discriminant $-D$, weighted by $1/2$ the order of their automorphism group, and $H(x)=0$ if $x\in\mathbb{Q}\setminus\mathbb{Z}$ (this follows the definitions in \cite{BKM,OR}; for more details, see \cite{Cox}, for example).

The following theorem gives a formula for the number of self-conjugate $7$-core partitions of certain odd integers in terms of Hurwitz class numbers.

\begin{theorem}[Ono--Raji, Theorem 1 of \cite{OR}]\label{7core_thm1}
If $n\not\equiv5\pmod{7}$ is a positive odd integer, then the number of self-conjugate $7$-core partitions of size $n$ is given by
\begin{equation}\label{7core_count1}
\mathrm{sc}_7(n)=\begin{cases}
\tfrac{1}{4}H(-28n-56),&\mbox{if }n\equiv1\pmod{4},\\
\tfrac{1}{2}H(-7n-14),&\mbox{if }n\equiv3\pmod{8},\\
0,&\mbox{if }n\equiv7\pmod{8}.
\end{cases}
\end{equation}
\end{theorem}

The next result expands Theorem \ref{7core_thm1} to count the number of self-conjugate $7$-core partitions of any integer $n$ as a linear combination of Hurwitz class numbers.

\begin{theorem}[Bringmann--Kane--Males, Theorem 1.3 of \cite{BKM}]\label{7core_thm2}
For every positive integer $n$, the number of self-conjugate $7$-core partitions of size $n$ is given by
\begin{equation}\label{7core_count2}
\mathrm{sc}_7(n)=\frac{1}{4}\left(H(-28n-56)-H\left(\frac{-4n-8}{7}\right)-2H(-7n-14)+2H\left(\frac{-n-2}{7}\right)\right).
\end{equation}
\end{theorem}

All of these values of $\mathrm{sc}_t(n)$ for $t=2,3,7$ will lead to results on the cardinalities of certain subsets of natural numbers in Section \ref{sec_Hurwitz}, although the subsets for $t=2,3$ are trivial.

%%%%%%%%%%%%%%%%%%%%%%%%%%%%%

\section{Proof of Theorem \ref{hook_thm}}\label{sec_proofs}

Suppose $\gamma$ is a self-conjugate partition and $\lambda=\left(\lambda_1,\lambda_2,\dots,\lambda_k\right)$ is the corresponding partition into distinct odd parts.  In particular, since $h_\gamma(i,i)=\lambda_i$ for all $1\leq i\leq k$, each leg of the hook with corner located at the main diagonal box $(i,i)$ contains $\frac{1}{2}\left(\lambda_i-1\right)$ boxes, excluding the corner box.  We will use this observation repeatedly to prove the hook length formula for $h_\gamma(i,j)$ for all boxes $(i,j)$ in the Young diagram of $\gamma$.  Throughout the proof, we assume without loss of generality (by symmetry) that the row $i$ is greater than or equal to the column $j$.

First, fix a row $i$ and a column $j$ in the Durfee square of the Young diagram of $\gamma$, so that $1\leq j\leq i\leq k$.  The total number of boxes in the $j$th column is $\frac{1}{2}\left(\lambda_j+1\right)+j-1$, so the number of boxes in the $j$th column below the box $(i,j)$ is $\frac{1}{2}\left(\lambda_j+1\right)+j-1-i$.  The number of boxes to the right of the box $(i,j)$ is the same as the number of boxes below the box $(j,i)$ by symmetry, which is $\frac{1}{2}\left(\lambda_i+1\right)+i-1-j$.  Combining these two legs and adding 1 to include the box at $(i,j)$ itself, we have that the hook lengths within the Durfee square satisfy
\begin{align*}
h_\gamma(i,j)&=\left(\frac{1}{2}\left(\lambda_j+1\right)+j-1-i\right)+\left(\frac{1}{2}\left(\lambda_i+1\right)+i-1-j\right)+1\\
&=\frac{1}{2}\left(\lambda_i+\lambda_j\right).
\end{align*}
This completes the proof of the formula for $h_\gamma(i,j)$, $i\leq k$.

Now, fix a row $i$ and a column $j$ below the Durfee square of the Young diagram of $\gamma$, so that $1\leq j\leq k<i$.  The number of boxes below the box $(i,j)$ is still $\frac{1}{2}\left(\lambda_j+1\right)+j-1-i$.  The number of boxes at and to the right of the box $(i,j)$ is equal to the number of columns at and to the right of column $j$ which contain at least $i$ boxes.  Since the $m$th column contains $\frac{1}{2}\left(\lambda_m+1\right)+m-1$ boxes for $1\leq m\leq k$, the $m$th column contains at least $i$ boxes if $\frac{1}{2}\left(\lambda_m+1\right)+m-1\geq i$.  We simplify this requirement as follows:
\begin{align}
\frac{1}{2}\left(\lambda_m+1\right)+m-1&\geq i\nonumber\\
\lambda_m+1+2m-2&\geq2i\nonumber\\
\lambda_m&\geq 1+2i-2m.\label{iboxes}
\end{align}
Therefore, the number of boxes at and to the right of the box $(i,j)$ is equal to the number of integers $m$, $j\leq m\leq k$, such that $\lambda_m\geq 1+2i-2m$.  All together, we have that the hook lengths below the Durfee square satisfy $$h_\gamma(i,j)=\frac{1}{2}\left(\lambda_j+1\right)+j-1-i+\#\left\{j\leq m\leq k:\lambda_m\geq 1+2i-2m\right\}.$$  This completes the proof of the formula for $h_\gamma(i,j)$, $i>k$.  The formula for $h_\gamma(j,i)$ follows by symmetry.\hfill$\square$

\section{Applications}\label{sec_Hurwitz}

Most of what number theorists know about partitions at this point in time arises from studying their additive properties, including the \emph{partition function} $p(n)$, which counts the number of partitions of size $n$, and its generating function.  It is also possible to study partitions from a multiplicative standpoint, however.  Recent work by the first author, Just, and Schneider \cite{DJS} introduces a multiplicative partition statistic called the \emph{supernorm} of a partition, denoted by $\widehat{N}(\lambda)$ for a partition $\lambda$.  The supernorm map $\widehat{N}$ from the set of all partitions to the set of all positive integers is defined as follows: $$\widehat{N}:\lambda=\left\langle1^{m_1},2^{m_2},3^{m_3},\dots,k^{m_k}\right\rangle\mapsto\prod_{i=1}^kp_i^{m_i},$$ where $p_i$ is the $i$th prime ($p_1=2$, $p_2=3$, $p_3=5$,\dots).  In other words, $\widehat{N}$ maps each partition $\lambda$ to the unique positive integer whose prime factors are indexed by the parts of $\lambda$.  Uniqueness of prime factorizations of positive integers makes this map a bijection.  For example, applying the supernorm map to the partition $$\lambda=(6,5,3,3,2,1,1,1)$$ yields the positive integer $$\widehat{N}(\lambda)=p_1^3\,p_2\,p_3^2\,p_5\,p_6=2^3\cdot3\cdot5^2\cdot11\cdot13=85800,$$ and no other partition corresponds to the number 85800 via $\widehat{N}$.

Many new applications arise from $\widehat{N}$, simply because it offers a new perspective with which to study partitions.  For an application to arithmetic densities, see \cite{DJS}.  One can also apply the supernorm map to a partition bijection in order to learn new information about certain (sometimes obscure) subsets of natural numbers.  The following example of $\widehat{N}$ applied to both sides of Euler's famous partition bijection is given in \cite{DJS}.  Note that a \emph{squarefree} positive integer is one whose prime factors are all distinct.

\begin{ex}
\emph{There is a bijection between squarefree positive integers and integers with only odd-indexed prime factors.}

Explicitly, if $n$ is a positive integer, then the number of squarefree positive integers whose prime factors' indices sum to $n$ is equal to the number of positive integers with only odd-indexed prime factors whose indices sum to $n$.
\end{ex}

The partition bijection in \eqref{bijection} with which this paper is concerned, between self-conjugate partitions and partitions into distinct odd parts, is not inherently able to give new information when factored through the supernorm map, because $\widehat{N}$ requires at least some details about the parts of partitions.  The partitions into distinct odd parts can be mapped to squarefree positive integers with only odd-indexed prime factors, but the parts of self-conjugate partitions generally do not follow a simple pattern and therefore may not have a clear correspondence with a specific subset of the positive integers.  Perhaps one could use \cite[Theorem 6.1]{Rebecca}, which gives necessary and sufficient conditions for a partition to be self-conjugate based only on the form of the parts, to make some progress in this area.

However, the set of $t$-core self-conjugate partitions is in bijection both with a specific subset of partitions into distinct odd parts for general $t$, as shown in Corollary \ref{corollary}, and with a specific subset of the positive integers for $t=2$, 3, and 7, as shown in Section \ref{section_hooks}.  One can therefore apply the supernorm map to both sides of these bijections to obtain a new subset of the positive integers which has cardinality given by the counts in equations \eqref{2core_count}, \eqref{3core_count}, \eqref{7core_count1}, and \eqref{7core_count2}.  The following examples make this correspondence explicit.

\begin{ex}[$2$-cores]\label{2core_ex}
The set of $2$-core partitions is the same as the set of perfectly triangular partitions $\gamma=(k,k-1,k-2,\dots,3,2,1)$, $k\geq1$, which are all self-conjugate, by Theorem \ref{2core_thm}.  Since we already have a description of the parts of all $2$-core partitions, we can apply the supernorm map to those partitions and also to the corresponding partitions into distinct odd parts: $\lambda=(2k-1,2k-5,2k-9,\dots,5,1)$ if $k$ is odd, or $\lambda=(2k-1,2k-5,2k-9,\dots,7,3)$ if $k$ is even.  The set $\left\{\widehat{N}(\gamma)\right\}$ for all perfectly triangular partitions $\gamma$ is the set of positive integers whose prime factorizations consist of all of the consecutive primes up to a certain point.  The set $\left\{\widehat{N}(\lambda)\right\}$ for all corresponding partitions $\lambda$ into distinct odd parts is the set of positive integers whose prime factors are all of the primes whose indices are $1\pmod{4}$ (resp. $3\pmod{4}$) up to a certain point.  Therefore, the supernorm map yields the following bijection between subsets of the positive integers:

\emph{There is a bijection between the set of positive integers whose prime factorizations consist of consecutive distinct primes starting from $p_1=2$ and the set of positive integers whose prime factorizations consist of consecutive distinct primes whose indices are $1\pmod{4}$ starting from $p_1=2$, or $3\pmod{4}$ starting from $p_3=5$.}

Explicitly, if $n$ and $k$ are fixed positive integers such that $n=k(k+1)/2$, then the number of positive integers with prime factorizations of the form $p_1\,p_2\,p_3\,\cdots\,p_k$ is equal to the number of squarefree positive integers whose prime factors consist of the first $k$ primes whose indices are $1\pmod{4}$ if $k$ is odd (resp. $3\pmod{4}$ if $k$ is even) and whose indices sum to $n$.

Notice that for each fixed $n$, these sets either have one element (if $n$ is a triangular number) or are empty by equation \eqref{2core_count}, so the explicit bijection which calculates the cardinality of each set for $2$-core partitions is basically trivial.
\end{ex}

In order to apply $\widehat{N}$ to the partitions provided in Theorem \ref{3core_thm2} and Theorems  \ref{7core_thm1} and \ref{7core_thm2}, we first identify the images under $\widehat{N}$ of the partitions $\lambda$ into distinct odd parts which correspond to the self-conjugate $3$-core and $7$-core partitions.  We fix positive integers $n$ and $t$ and apply $\widehat{N}$ to the set of all restricted partitions $\lambda=\left(\lambda_1,\lambda_2,\dots,\lambda_k\right)$ defined in Corollary \ref{corollary} to obtain the set of all squarefree positive integers with prime factorizations $\prod_{1\leq i\leq k}p_{\lambda_i},$ whose indices $\lambda_i$ are all odd and sum to $n$, such that $$\frac{1}{2}\left(\lambda_i+\lambda_j\right)\not\equiv0\pmod{t}$$ for all pairs of integers $(i,j)$ with $1\leq j\leq i\leq k$, and $$\frac{1}{2}\left(\lambda_j+1\right)+j-i-1+\#\left\{j\leq m\leq k:\lambda_m\geq 2i-2m+1\right\}\not\equiv0\pmod{t}$$ for all pairs of integers $(i,j)$ with $1\leq j\leq k<i\leq\frac{1}{2}\left(\lambda_1+1\right)$ and $\lambda_j\geq 2i-2j+1$.  Although this description is fairly difficult to use in practice, it does give a well-defined subset of the positive integers.

\begin{ex}[$3$-cores]\label{3core_ex}
The set of self-conjugate $3$-core partitions is the same as the set of partitions of the form $$\gamma_-=\left\langle1^2,2^2,\dots,(r-2)^2,(r-1)^2,r,r+2,\dots,3r-4,3r-2\right\rangle,$$ if $n=r(3r-2)$ for some $r\geq1$, or $$\gamma_+=\left\langle1^2,2^2,\dots,(r-2)^2,(r-1)^2,r^2,r+2,\dots,3r-2,3r\right\rangle,$$ if $n=r(3r+2)$ for some $r\geq1$, by Theorem \ref{3core_thm2}.  The set $\left\{\widehat{N}\left(\gamma_-\right)\right\}$ for all such partitions $\gamma_-$ is the set of positive integers with prime factorizations of the form $$2^2\cdot3^2\cdots p_{r-2}^2\cdot p_{r-1}^2\cdot p_r\cdot p_{r+2}\cdots p_{3r-4}\cdot p_{3r-2}$$ for some $r\geq1$; the set $\left\{\widehat{N}\left(\gamma_+\right)\right\}$ for all such partitions $\gamma_+$ is the set of positive integers with prime factorizations of the form $$2^2\cdot3^2\cdots p_{r-1}^2\cdot p_r^2\cdot p_{r+2}\cdots p_{3r-2}\cdot p_{3r}$$ for some $r\geq1$.  By Corollary \ref{corollary}, the set of self-conjugate $3$-core partitions is in bijection with a specific set of partitions $\lambda$ into distinct odd parts whose images under $\widehat{N}$ yield the set of positive integers described above for $t=3$.  Therefore, the supernorm map yields the following bijection between subsets of the positive integers:

\emph{There is a bijection between the set of positive integers with prime factorizations of the form $$2^2\cdot3^2\cdots p_{r-2}^2\cdot p_{r-1}^2\cdot p_r\cdot p_{r+2}\cdots p_{3r-4}\cdot p_{3r-2}$$ for some $r\geq1$ or $$2^2\cdot3^2\cdots p_{r-1}^2\cdot p_r^2\cdot p_{r+2}\cdots p_{3r-2}\cdot p_{3r}$$ for some $r\geq1$ and the set of squarefree positive integers whose prime factors' indices $\lambda_i$ are odd, such that $$\frac{1}{2}\left(\lambda_i+\lambda_j\right)\not\equiv0\pmod{3}$$ for all pairs of integers $(i,j)$ with $1\leq j\leq i\leq k$, and $$\frac{1}{2}\left(\lambda_j+1\right)+j-i-1+\#\left\{j\leq m\leq k:\lambda_m\geq 2i-2m+1\right\}\not\equiv0\pmod{3}$$ for all pairs of integers $(i,j)$ with $1\leq j\leq k<i\leq\frac{1}{2}\left(\lambda_1+1\right)$ and $\lambda_j\geq 2i-2j+1$.}

Notice that for each fixed $n$, these sets either have one element (if $n=r(3r\pm2)$ for some $r\geq1$) or are empty by equation \eqref{3core_count}, so the corresponding explicit bijection within the set of positive integers for $3$-core partitions is basically trivial as well.
\end{ex}

\begin{ex}[$7$-cores]\label{7core_ex}
There is not an explicit description of the parts of self-conjugate $7$-core partitions, so we only have the above description of the images under $\widehat{N}$ of the corresponding partitions $\lambda$ into distinct odd parts for $t=7$.  However, equations \eqref{7core_count1} and \eqref{7core_count2} give the number of self-conjugate $7$-core partitions of a positive integer as a Hurwitz class number or a linear combination of Hurwitz class numbers.  Therefore, these (linear combinations of) Hurwitz class numbers also give the number of partitions $\lambda$ described above, which is the same as the number of elements in the set $\left\{\widehat{N}(\lambda)\right\}$ for all such $\lambda$.  We have the following two new expressions for these Hurwitz class number formulas.
\begin{enumerate}
\item \emph{Fix a positive odd integer $n\not\equiv5\pmod{7}$.  If $n\equiv1\pmod{4}$ (resp. $n\equiv3\pmod{8}$), then $\frac{1}{4}H(-28n-56)$ (resp. $\frac{1}{2}H(-7n-14)$) is equal to the number of squarefree positive integers with prime factorizations $\prod_{1\leq i\leq k}p_{\lambda_i},$ whose indices $\lambda_i$ are all odd and sum to $n$, such that $$\frac{1}{2}\left(\lambda_i+\lambda_j\right)\not\equiv0\pmod{7}$$ for all pairs of integers $(i,j)$ with $1\leq j\leq i\leq k$, and $$\frac{1}{2}\left(\lambda_j+1\right)+j-i-1+\#\left\{j\leq m\leq k:\lambda_m\geq 2i-2m+1\right\}\not\equiv0\pmod{7}$$ for all pairs of integers $(i,j)$ with $1\leq j\leq k<i\leq\frac{1}{2}\left(\lambda_1+1\right)$ and $\lambda_j\geq 2i-2j+1$.  If $n\equiv7\pmod{8}$, then there are no such positive integers.}
\item \emph{Fix any positive integer $n$.  Then $$\frac{1}{4}\left(H(-28n-56)-H\left(\frac{-4n-8}{7}\right)-2H(-7n-14)+2H\left(\frac{-n-2}{7}\right)\right)$$ is equal to the number of squarefree positive integers with prime factorizations $\prod_{1\leq i\leq k}p_{\lambda_i},$ whose indices $\lambda_i$ are all odd and sum to $n$, such that $$\frac{1}{2}\left(\lambda_i+\lambda_j\right)\not\equiv0\pmod{7}$$ for all pairs of integers $(i,j)$ with $1\leq j\leq i\leq k$, and $$\frac{1}{2}\left(\lambda_j+1\right)+j-i-1+\#\left\{j\leq m\leq k:\lambda_m\geq 2i-2m+1\right\}\not\equiv0\pmod{7}$$ for all pairs of integers $(i,j)$ with $1\leq j\leq k<i\leq\frac{1}{2}\left(\lambda_1+1\right)$ and $\lambda_j\geq 2i-2j+1$.}
\end{enumerate}
\end{ex}

Although it is an understatement to say that the explicit bijections and formulas for Hurwitz class numbers given in Examples \ref{2core_ex}, \ref{3core_ex}, and \ref{7core_ex} are bulky and difficult to use, they do give ``nice" combinatorial expressions which were previously unknown, where by ``nice" we mean that they actually count elements of explicit sets of positive integers.  These counts are essentially trivial for the simplest cases of $t=2$ and $t=3$ but more interesting for $t=7$.  Perhaps if results on self-conjugate $t$-core partitions are proven for other values of $t$, the counts in these cases would also lead to interesting (albeit cumbersome) formulas for other known number theoretic quantities.

Another direction that could potentially lead to further applications is the study of $t$-cores, self-conjugate $t$-cores, and sums of squares in \cite{MT}.  Males and Tripp construct explicit bijections between families of $t$-cores and certain sets of representations of integers as sums of squares dictated by congruence conditions.  As an application, for instance, one could attempt to combine the connections provided here between self-conjugate $t$-cores and the partitions into distinct odd parts detailed in Corollary \ref{corollary} with the connections provided in \cite{MT} to obtain new results on sums of squares.

%%%%%%%%%%%%%%%%%%%%%%%%%%%%%

\section*{Acknowledgments}
The authors would like to thank Matthew Just and Robert Schneider for many enlightening discussions regarding partition statistics, and Ken Ono for his helpful comments on the content of this paper.  The authors also thank the referees for many suggestions which improved the paper.

%%%%%%%%%%%%%%%%%%%%%%%%%%%%%

%%%%%%%%%%%%%%%%%%%%%%%%%%%%%

\end{document}